\documentclass[12pt]{amsart}
\usepackage{amssymb,amsmath,amsfonts,epsfig,textcomp,amsthm}
\usepackage{enumerate}
\usepackage[enableskew]{youngtab}
\usepackage{graphics}
\usepackage[all,cmtip]{xy}
\usepackage{tikz}
\usetikzlibrary{patterns}
\usepackage{color}

\newtheorem{theorem}{Theorem}[section]
\newtheorem{corollary}[theorem]{Corollary}
\newtheorem{lemma}[theorem]{Lemma}
\newtheorem{proposition}[theorem]{Proposition}

\newtheorem{remark}[theorem]{Remark}

\theoremstyle{definition}
\newtheorem{definition}[theorem]{Definition}
\newtheorem{example}[theorem]{Example}

\DeclareMathOperator{\Comp}{Comp}
\DeclareMathOperator{\Des}{Des}

\DeclareMathOperator{\type}{type}

\numberwithin{equation}{section}
\begin{document}
\title[colored eulerian polynomials and permutohedra]{colored eulerian polynomials and the Colored permutohedron }
\author{Dustin Hedmark}
\date{\today}
\address{Department of Mathematics\\University of Kentucky}
\email{dustin.hedmark@uky.edu}
\begin{abstract}
This paper introduces a colored generalization of the Eulerian polynomials, denoted the $\alpha$-colored Eulerian polynomials. We first compute these polynomials by taking the $h$-vector of the $\alpha$-colored permutohedron, a colored analog of the permutohedron which we develop. We also arrive at the $\alpha$-colored Eulerian polynomials combinatorially by defining a correct notion of descent for colored permutations.
\end{abstract}
\maketitle
\section{Introduction}
The Eulerian polynomials have a long and rich history in combinatorics. Euler first defined the Eulerian polynomials as the numerator for the generating function of the $n$'th powers, that is, the degree $n$ polynomial satisfying $A_n(x)=(1-x)^{n+1}\sum_{k\geq 0}k^nx^k$. Equivalently, the Eulerian polynomials can be defined as a sum over descents in the symmetric group $\mathfrak{S}_n$, namely as $A_n(x)=\sum_{\pi\in\mathfrak{S}_n}x^{1+d(\pi)}$, for $d(\pi)$ the number of descents of $\pi$. Yet another way to arrive at the Eulerian polynomials is as the $h$-polynomial of the permutohedron $P_n$.

In this paper we define the $\alpha$-colored Eulerian polynomials for $\alpha$ a positive integer. We develop an $\alpha$-colored analog for the descent and permutohedron constructions of the usual Eulerian polynomials mentioned above, but we arrive at our polynomials in a different order than was done classically. Namely, we first construct the $\alpha$-colored Eulerian polynomials in Section~\ref{section_polynomial} as the $h$-polynomials of the $\alpha$-colored permutohedron, $P_n^\alpha$, which we develop in Section~\ref{alpha_complex_section}. In Section~\ref{section_colored_permutations} we develop the combinatorial theory of $\alpha$-colored permutations and express the $\alpha$-colored Eulerian polynomials in terms of descents. Lastly, we give a recurrence relation for the $\alpha$-Eulerian numbers which is derived using the combinatorial theory developed in Section~\ref{section_colored_permutations}.

As a generalization of the usual Eulerian polynomials, the $\alpha$-colored Eulerian polynomials are real rooted, and hence log-concave, as will be shown in Section~\ref{section_polynomial}. It is worth mentioning now that any results about $\alpha$-colored constructions recover classical results by taking $\alpha=1$.
\section{preliminaries}
We begin with a discussion of the lattice of compositions of $n$, denoted $\Comp(n)$, as well as a discussion of the ordered partition lattice, $Q_n$. Throughout the paper, we denote the $n$-set by $[n]=\{1,2,\dots,n\}$.
\begin{definition}[$\Comp{(n)}$]
Let $\Comp(n)$ denote the poset of ordered integer partitions of $n$ into non-negative parts, with cover relation given by adding adjacent parts. The minimum and maximum elements of $\Comp{(n)}$ are $\hat{0}=(1,1,\dots,1)$ and $\hat{1}=(n)$ respectively. For a composition $\vec{c}=(c_1,c_2,\dots,c_k)$ we refer to $c_i$ as the $i$'th part of $\vec{c}$.
\end{definition}
Let $Q_n$ denote the poset of ordered set partitions on $n$ objects. In contrast to the usual partition lattice $\Pi_n$, order among the blocks in $Q_n$ matters. The cover relation in $Q_n$ is given by the merging of adjacent blocks.

The \emph{type} of an ordered set partition $\tau=(B_1,B_2,\dots,B_k)$ in $Q_n$ is defined to be the composition of $n$ given by the cardinality of the blocks of $\tau$ in order, or $\type(\tau)=(|B_1|,|B_2|,\dots,|B_k|)$ in $\Comp(n)$.
\begin{definition}[$\alpha$-colored partition lattice] 
\label{Q_n_alpha}
Let $Q_n^\alpha$ be the collection of ordered set partitions where each block has one of $\alpha$ colors, with the last block a fixed color. The cover relationship in $Q_n^\alpha$ is given by the merging of adjacent blocks of the same color.
\end{definition}
 There is a clear bijection between elements of $Q_n^\alpha$ and ordered set partitions where we color the breaks between the blocks of the partition, namely by coloring bars between blocks the color of the block to its left and forgetting the color of the last block. This bijection only holds on the level of elements--we are not asserting a poset structure on the set of ordered set partitions with colored bars.
%In Theorem~\ref{theorem_main} we count elements of $Q_n^\alpha$ with the colored bar interpretation, thus an example of how we interpret elements of $Q_n^\alpha$ as having colored bars instead of colored blocks is given now.

\begin{example}
~\label{red_green}
Let $n=5$ and $\alpha=2$. Instead of having two colors, we will let our blocks be hatted or bald, and force our last block to be hatted. Four distinct elements of $Q_5^2$ are $\hat{1}\hat{2}\hat{3}|\hat{4}\hat{5}$, $123|\hat{4}\hat{5}$, $\hat{4}\hat{5}|\hat{1}\hat{2}\hat{3}$ and $45|\hat{1}\hat{2}\hat{3}$.

Alternatively, the elements of $Q_n^\alpha$ can be thought of as having colored bars $|$ and $\hat{|}$, yielding the respective elements
 $123\hat{|}45, 123|45, 45\hat{|}123,$ and $45|123$.
\end{example}
While $Q_n^\alpha$ is primarily introduced as a means to develop the $\alpha$-colored permutohedron $P_n^\alpha$, we remark that just as ordered set partitions are an important tool in the computation of the composition of ordinary generating functions, the poset $Q_n^\alpha$ is an indexing poset for the $n$-fold composition of ordinary generating functions. This suggests another interpretation of $Q_n^\alpha$.

When $\alpha=1$, $Q_n^1$ is the usual ordered partition lattice, which we think of as lists of sets. When $\alpha=2$, with ``colors" $|$ and $\hat{|}$ as in Example~\ref{red_green}, then $Q_n^2$ can be thought of as lists of lists of sets, where $\hat{|}$ denotes a comma in an outer list and $|$ denotes a comma in an inner list. Continuing in this fashion, we can think of $Q_n^3$ as lists of lists of lists of sets, and so on. The use of the terminology lists of lists of sets comes from Motzkin's paper~\cite{Motzkin}. This correspondence is demonstrated in the following example.
\begin{example}
Let $n=5$ and $\alpha=2$, with bars $|$ and $\hat{|}$. Then:
\begin{align*}
1\hat{|}23|45&\longleftrightarrow \{1,\{23,45\}\}\\ 45|1\hat{|}2{\color{green}}3&\longleftrightarrow\{\{45,1\},23\}\\
45|1|2{\color{red}}3&\longleftrightarrow\{\{45,1,23\}\}\\
\end{align*}
\end{example}

\begin{proposition}
The exponential generating function for the $\alpha$-colored ordered partition lattice is given by 
$$\sum_{n\geq 1}\frac{|Q_n^\alpha|}{n!}x^n=\frac{e^x-1}{1-\alpha(e^x-1)}.$$
\end{proposition}
\begin{proof}
We use the composition principle of exponential generating functions.
The $\alpha$-colored ordered partition lattice can be described as a composition of two structures on the set $[n]$, namely, an inner non-empty structure given by $e^x-1$, and an outer $\alpha$-colored permutation structure with generating function given by ~$x/(1-\alpha\cdot x)$. 
\qedhere
\end{proof}
\section{$\alpha$-colored ordered set partitions, Eulerian polynomials, and the permutohedron.}
In this section we demonstrate the close relationship between the Eulerian polynomials and the permutohedron. By means of Theorem~\ref{theorem_main}, we show that the Eulerian polynomial computes the Euler characteristic of the permutohedron. The key ingredient to the proof is that the face lattice of the permutohedron $P_n$ is the ordered set partition lattice $Q_n$. In Section~\ref{alpha_complex_section} we will mirror the analogy between the permutohedron and the Eulerian polynomial with a new polytopal complex which we call the $\alpha$-colored permutohedron, and subsequently, with the $\alpha$-colored Eulerian polynomials in Section~\ref{section_polynomial}. We now proceed with the definition of descents in the symmetric group $\mathfrak{S}_n$ and of the permutohedron.

For a permutation $\pi\in\mathfrak{S}_n$, the descent set of $\pi$ is given by $D(\pi)=\{i\in[n-1]:\,\pi(i)>\pi(i+1)\}$. We let $d(\pi)$ be the number of descents of $\pi$, or $d(\pi)=|D(\pi)|$. It will often be more advantageous to think of the descent set of $\pi$ as a composition of $n$ in the usual way, and thus we define:
\begin{definition}
Let the descent set of $\pi$ in the symmetric group $\mathfrak{S}_n$ be given by $\{i_1,i_2,\dots,i_k\}$. We convert this descent set into the \emph{descent composition} of $\pi$ by
$D(\pi)=(i_1,i_2-i_1,\dots,i_k-i_{k-1},n-i_k)$.
\end{definition}
We now define the Eulerian polynomial, as in Chapter $1$ of ~\cite{EC1}.
\begin{definition}
$A_n(x)=\sum_{\pi\in\mathfrak{S}_n}x^{1+d(\pi)}$ is the \emph{Eulerian polynomial}.
\end{definition}
Recall that the \emph{permutohedron}, $P_n$, is the $(n-1)$-dimensional polytope obtained by taking the convex hull of the permutations of $\mathfrak{S}_n$ in $\mathbb{R}^n$. The $d$-dimensional faces of $P_n$ are in one to one correspondence with ordered set partitions of $n$ into $n-d$ parts. For example, the vertices of $P_n$ are
given by permutations in $\mathfrak{S}_n$, which are in bijection with ordered set partitions of $n$ into $n=n-0$ parts.  We note that the interior of $P_n$, of dimension $n-1$, must then be in bijection with ordered set partitions of $n$ into $n-(n-1)=1$ part. In other words, $P_n$ is contractible.

In parallel to Proposition~\ref{proposition_Euler_alpha} to come in Section~\ref{alpha_complex_section}, we now demonstrate that the Euler characteristic of the permutohedron $P_n$ can be computed with the Eulerian polynomial $A_n(x)$. We do this by counting the elements of the $\alpha$-colored ordered partition lattice $Q_n^\alpha$ with the Eulerian polynomial in Theorem~\ref{theorem_main}. Recall that the \emph{Stirling number of the second kind}, $S(n,k)$, is the number of partitions of $n$ into $k$ non-zero parts.
\begin{corollary}
\label{Euler_P_n}
The Euler characteristic of $P_n$ can be computed from the Eulerian polynomial $A_n(x)$.
\end{corollary}
\begin{proof}
Given that the face lattice of the permutohedron is the dual of the ordered set partition lattice yields
$$\chi(P_n)=\sum_{k=0}^{n-1}(-1)^k f_k(P_n)=\sum_{k=0}^{n-1}(-1)^kS(n,n-k)k!=1,$$ since $P_n$ is contractible.

Using Theorem~\ref{theorem_main} we have that:
\begin{align*}
\sum_{k=0}^n S(n,k)k!\alpha^{k-1}&=(\alpha+1)^n/\alpha\cdot A_n(\alpha/(\alpha+1))\\
&=(\alpha+1)^n/\alpha\cdot\sum_{\pi\in\mathfrak{S}_n}(\alpha/(\alpha+1))^{1+d(\pi)}\\
&= (\alpha+1)/\alpha \cdot L+\frac{(\alpha+1)^n}{\alpha}\cdot\frac{\alpha^n}{(\alpha+1)^n}\\
&=(\alpha+1)/\alpha \cdot L+\alpha^{n-1},
\end{align*}
where $L$ is all terms of $A_n(\alpha/(\alpha+1))$ except for the term corresponding to the permutation with $n$ descents. Lastly, let $\alpha=-1$ to obtain
$\sum_{k=0}^{n}S(n,k)k!(-1)^{k-1}=(-1)^{n-1}$. The latter sum is the Euler characteristic of $P_n$ when $n$ is odd, and when $n$ is even the Euler characteristic of $P_n$ is obtained by multiplying both sides of the latter sum by $-1$.
\end{proof}
\begin{theorem}
\label{theorem_main}
The following identity holds between the Eulerian polynomial and the Stirling numbers of the second kind:$$\frac{(\alpha+1)^n}{\alpha} A_n\left(\frac{\alpha}{\alpha+1}\right)=\sum_{k=0}^n S(n,k)k!\alpha^{k-1}.$$
\end{theorem}
\begin{proof}
We construct a map $P:Q_n^{\alpha}\longrightarrow\mathfrak{S}_n$ given by writing out the elements of each block of an $\alpha$-colored ordered set partition in increasing order, then considering this string as a permutation, in one line notation, in $\mathfrak{S}_n$. We will think of elements in $Q_n^\alpha$ as having bars with one of $\alpha$ colors, as discussed in Example~\ref{red_green}.

Let $\pi\in\mathfrak{S}_n$ such that $d(\pi)=k$. Any $\alpha$-colored ordered set partition in the fiber $P^{-1}(\pi)$ must have $\alpha$-colored breaks at the descents of $\pi$, but otherwise it is free to have $\alpha$-colored breaks at any position. Hence if the descent runs of $\pi$ have sizes $d_1,d_2,\dots, d_{d(\pi)+1}$, then we have that 
\begin{align*}
|P^{-1}(\pi)|&=(\alpha+1)^{d_1-1}(\alpha+1)^{d_2-1}\cdots(\alpha+1)^{d_{d(\pi)+1}-1}\alpha^{d(\pi)}\\
&=(\alpha+1)^{d_1+d_2+\dots+d_{k+1}-(k+1)}\alpha^{d(\pi)}\\
&=(\alpha+1)^{n-k-1}\alpha^{d(\pi)}\\
&=(\alpha+1)^{n-d(\pi)-1}\alpha^{d(\pi)}.
\end{align*}
Since $P$ is surjective we have that:
\begin{align*}
|Q_{n}^\alpha|&=\sum_{\pi\in\mathfrak{S}_n}|P^{-1}(\pi)|\\
&=\sum_{\pi\in\mathfrak{S}_n}(\alpha+1)^{n-d(\pi)-1}\alpha^{d(\pi)}\\
&=(\alpha+1)^{n}(1/\alpha) \sum_{\pi\in\mathfrak{S}_n}(\alpha+1)^{-d(\pi)-1}\alpha^{d(\pi)+1}\\
&=(\alpha+1)^{n} (1/\alpha) \sum_{\pi\in\mathfrak{S}_n}(\alpha/(\alpha+1))^{1+d(\pi)}\\
&=(\alpha+1)^n (1/\alpha)A_n(\alpha/(\alpha+1)).
\end{align*}
Finally, we note that the cardinality of $\alpha$-colored ordered set partitions is given by $|Q_n^{\alpha}|=\sum_{k=0}^n S(n,k)k!\alpha^{k-1}$, since the number of $\alpha$-colored ordered set partitions into $k$ blocks is counted by $k!\cdot S(n,k)$ times $\alpha^{k-1}$, with the latter term accounting for the $\alpha$ possible colors of the $k-1$ breaks.
\end{proof}
Note that Theorem~\ref{theorem_main} recovers Theorem $5.3$ of \cite{Petersen}, as the latter author defines the Eulerian polynomial as $A_n(x)=\sum_{\pi\in\mathfrak{S}_n}x^{d(\pi)}$. 
\begin{corollary} 
\label{stanley_result}
$2^n \cdot A_n(1/2)=|Q_n^1|=|Q_n|$.
\end{corollary}
The result follows by letting $\alpha=1$ in Theorem~\ref{theorem_main} and by the symmetry of the Eulerian polynomials. Additionally, Corollary~\ref{stanley_result} recaptures Exercise $33$ of Chapter $1$ of ~\cite{EC1}.

We now give a non-topological consequence of Theorem~\ref{theorem_main}.

Recall Euler's generating function definition for the Eulerian polynomials:
\begin{equation}
\label{n_power}
\sum_{k\geq 0}k^nx^k=\frac{A_n(x)}{(1-x)^{n+1}}.
\end{equation}
Using Theorem ~\ref{theorem_main} and Equation~\eqref{n_power} we obtain the following corollary.
\begin{corollary}
For $\alpha\in\mathbb{C}$ with $\operatorname{Re}(\alpha)>-1/2$ we have that
$$\sum_{k=0}^{\infty}k^n\left(\frac{\alpha}{\alpha+1}\right)^k=(\alpha+1)\cdot\alpha\sum_{k=0}^n S(n,k)k!\alpha^{k-1}.$$
\end{corollary}
\section{The construction of $P_n^\alpha$}
\label{alpha_complex_section}
 When $\alpha=1$, $Q_n^\alpha$ is the usual ordered partition lattice, where the number of elements in $Q_n$ of rank $i$ count the $(n-i)$-dimensional faces of the permutohedron, $P_n$. In this section we define an analogous polytopal complex $P_n^\alpha$ whose $(n-i)$-dimensional faces are counted by the elements of rank $i$ in $Q_n^\alpha$. We call $P_n^\alpha$ the \emph{$\alpha$-colored permutohedron}, see Definition~\ref{alpha_complex}.

Over the course of this section we will see that $P_n^\alpha$ has many similarities to the usual permutohedron $P_n$. Namely,  the face lattice of $P_n^\alpha$ is $Q_n^\alpha$, while the face lattice of $P_n$ is $Q_n$. Furthermore, Proposition~\ref{proposition_Euler_alpha} shows that the Euler characteristic of $P_n^\alpha$ can be computed with the Eulerian polynomial $A_n(x)$ in a similar fashion to Corollary~\ref{Euler_P_n}. Lastly, the components of $P_n^\alpha$ are products of permutohedra, and therefore $P_n^\alpha$ is a union of contractible components, just as $P_n$ is contractible.

Keeping in mind that the face lattice of $P_n^\alpha$ should be $Q_n^\alpha$, we reverse engineer the construction of $P_n^\alpha$ by describing the facets of the polytopal complex--each of which will determine a unique connected component of the complex.

Let $\vec{c}=(c_1,c_2,\dots,c_k)$ be a composition of $n$. Recall the \emph{multinomial coefficient} $\binom{n}{\vec{c}}=\binom{n}{c_1,c_2,\dots,c_k}$. Notice that there are ~$(\alpha-1)^{k-1}\cdot\binom{n}{\vec{c}}$ elements of $Q_n^\alpha$ of type $\vec{c}$ with no adjacent blocks of the same color. For each of these alternating color ordered partitions we define the facet $P_{|c_1|}\times P_{|c_2|}\times\dots\times P_{|c_k|}$ of $P_n^\alpha$. As the goal is to define a polytopal complex $P_n^\alpha$ with face lattice $Q_n^\alpha$, defining facets of $P_n^\alpha$ in this manner makes sense. This is because an $\alpha$-colored ordered set partition $\tau$ of type $\vec{c}$ with alternating colors is not covered by any element in $Q_n^\alpha$ per Definition~\ref{Q_n_alpha}, and thus should correspond to a facet of $P_n^\alpha$. 

Moreover, any element $\tau'\in Q_n^\alpha$ with $\tau'\leq\tau$ can be formed by splitting blocks of $\tau$ and by  flipping blocks of $\tau$ of the same color. This splitting and flipping can be done independently, and since the faces of the usual permutohedron $P_n$ are enumerated by the ordered partition lattice $Q_n$, the face lattice of the component $P_{|c_1|}\times P_{|c_2|}\times\dots\times P_{|c_k|}$ will be the lower order ideal generated by its defining alternating color ordered set partition $\tau$ of type $\vec{c}$ in $Q_n^\alpha$. Since our components are disjoint, this construction will yield $f(P_n^\alpha)=Q_n^\alpha$, and gives the following definition.
\begin{definition}[$\alpha$-colored permutohedron]
\label{alpha_complex}
Let $P_n^\alpha$ be the polytopal complex with $(\alpha-1)^{k-1}\cdot\binom{n}{\vec{c}}$ disjoint facets $P_{|c_1|}\times P_{|c_2|}\times\dots\times P_{|c_k|}$ for each composition $\vec{c}=(c_1,c_2,\dots,c_k)$ of $n$. Each of these facets is labeled by a unique $\alpha$-colored ordered partition, $\tau$, of type $\vec{c}$ with no adjacent blocks of the same color.
The lower dimensional faces of codimension $i$ in the facet labeled by $\tau$ are labeled by $\alpha$-colored ordered set partitions in the lower order ideal generated by $\tau$ in $Q_n^\alpha$ into $|\tau|+i$ parts. By virtue of construction, we have that the face lattice of $P_n^\alpha$ is given by $Q_n^\alpha$, that is $f(P_n^\alpha)=Q_n^\alpha$.
\end{definition}
We now look at an example.
\begin{example}
{\rm
Figure~\ref{blue_red} shows $P_3^2$. Per Definition~\ref{alpha_complex}, $P_3^2$ has facets labeled by $2$-colored ordered set partitions with alternating colors, which we mark as bald blocks and hatted blocks. Instead of a fixed last color, we force the last block of each $2$-colored ordered set partition to be hatted. We now compute the facets of $P_3^2$ using Lemma~\ref{components_p_n}. The compositions of three are $(1,1,1),(1,2),(2,1)$ and $(3)$. Since $\binom{3}{1,1,1}=6$, $\binom{3}{2,1}=3$, $\binom{3}{1,2}=3$ and $\binom{3}{3}=1$, there are $6$ facets of type $(1,1,1)$, $3$ facets of type $(2,1)$, $3$ facets of type $(1,2)$, and $1$ facet of type $(3)$. Each of these facets must be alternating in color with last color hatted.

}\end{example}
We now list relevant topological properties of $P_n^\alpha$.
\begin{lemma}
\label{components_p_n}
The number of components of $P_n^\alpha$ is given by
\begin{equation}
\label{eq:comp}
\sum_{\vec{c}\in\Comp{(n)}}(\alpha-1)^{|\vec{c}\,|-1}\cdot\binom{n}{\vec{c}}.
\end{equation}
\end{lemma} 
The above sum enumerates all elements of $Q_n^\alpha$ with adjacent blocks having different colors. Since each of these partitions determine a component, the result follows.
\begin{corollary}
The number of connected components of $P_n^\alpha$ is the total number of faces in $P_n^{\alpha-1}$.
\end{corollary}
We can view equation \eqref{eq:comp}  as summing over all $(\alpha-1)$-colored ordered set partitions of type $\vec{c}$, and the result follows from Lemma~\ref{components_p_n}.
\begin{corollary}
\label{euler_alpha_perm}
The Euler characteristic $\chi(P_n^\alpha)$ can be counted in two ways as
\begin{equation}
\label{eq:Eulerian_des}
\sum_{k=0}^n (-1)^k S(n,n-k)\cdot(n-k)!\cdot\alpha^{n-k-1}=\sum_{\vec{c}\in\Comp{(n)}}(\alpha-1)^{|\vec{c}\,|-1}\cdot\binom{n}{\vec{c}}.
\end{equation}
\end{corollary}
Since $P_n^\alpha$ is a union of contractible components, Lemma~\ref{components_p_n} counts the Euler characteristic of $P_n^\alpha$. We can also count the Euler characteristic by the alternating sum of the face numbers of $Q_n^\alpha$, which is the left hand side of \eqref{eq:Eulerian_des}.

Note that letting $\alpha=1$ in Corollary~\ref{euler_alpha_perm} recaptures that the usual permutohedron is contractible, and thus has Euler characteristic $1$.

Lastly, to complete the analogy between $P_n^\alpha$ and the usual permutohedron $P_n$, we show that the Eulerian polynomial can also be used to compute the Euler characteristic $\chi(P_n^\alpha)$, just as we showed the Eulerian polynomial can be used to compute $\chi(P_n)$ in Corollary \ref{Euler_P_n}.
\begin{proposition}
\label{proposition_Euler_alpha}
The Euler characteristic of $P_n^\alpha$ is given by 
$$\chi(P_n^\alpha)=\frac{(\alpha-1)^n}{\alpha}A_n\left(\frac{\alpha}{\alpha-1}\right).$$
\end{proposition}
\begin{proof}
This proof will mimic the proof of Theorem~\ref{theorem_main}.

Consider the restriction of the map $P:Q_n^\alpha\longrightarrow\mathfrak{S}_n$ to partitions alternating in color. We denote this map $P_A$. Since partitions alternating in colors are the facets of $P_n^{\alpha}$, the sum $\sum_{\pi\in\mathfrak{S}_n}|P_A^{-1}(\pi)|$ will give us our desired Euler characteristic.

For a fixed $\pi\in\mathfrak{S}_n$, the size of the fiber $P_A^{-1}(\pi)$ is given by
$$|P_A^{-1}(\pi)|=\sum_{\vec{d}\,\leq D(\pi)}(\alpha-1)^{|\vec{d}\,|-1},$$
as we are allowed to add breaks between descents of $\pi$ while maintaining alternating colors.

Therefore,
\begin{align*}
\chi(P_n^\alpha)&=\sum_{\pi\in\mathfrak{S}_n}|P_A^{-1}(\pi)|\\
&=\sum_{\pi\in\mathfrak{S}_n}\sum_{\vec{d}\,\leq D(\pi)}(\alpha-1)^{|\vec{d}\,|-1}\\
&=(\alpha-1)^{n-1}\sum_{\pi\in\mathfrak{S}_n}\sum_{\vec{d}\,\leq D(\pi)}(1/(\alpha-1))^{n-|\vec{d}\,|}.
\end{align*}
Notice that the last sum $\sum_{\vec{d}\,\leq D(\pi)}(\alpha-1)^{n-|\vec{d}\,|}$ is the rank generating function for the lower order ideal generated by $D(\pi)=\vec{c}=(c_1,c_2,\dots,c_k)$ in $\Comp(n)$, evaluated at $\alpha-1$. As lower order ideals in $\Comp(n)$ are products of Boolean algebras, we may express this inner sum as a product of rank generating functions for corresponding Boolean algebras:
\begin{align*}
\chi(P_n^\alpha)&=(\alpha-1)^{n-1}\sum_{\pi\in\mathfrak{S}_n}\sum_{\vec{d}\,\leq D(\pi)}(1/(\alpha-1))^{n-|\vec{d}\,|}\\
&=(\alpha-1)^{n-1}\sum_{\pi\in\mathfrak{S}_n}\prod_{i=1}^{k}F_{B_{c_i}}(1/(\alpha-1))\\
&=(\alpha-1)^{n-1}\sum_{\pi\in\mathfrak{S}_n}\prod_{i=1}^{k}(1+1/(\alpha-1))^{c_i-1}\\
&=(\alpha-1)^{n-1}\sum_{\pi\in\mathfrak{S}_n}\left(\frac{\alpha}{\alpha-1}\right)^{n-k}\\
&=\frac{(\alpha-1)^n}{\alpha}A_n\left(\frac{\alpha}{\alpha-1}\right).
\end{align*}
Observe we have used that the rank generating function of the Boolean Algebra $B_n$ is given by $F_{B_n}(x)=(1+x)^{n-1}$. While some steps at the end of the calculation have been omitted, the reader may see Theorem~\ref{theorem_main} for similar reasoning.
\end{proof}
\begin{figure}
\caption{$P_3^2$ with ``colors" bald and hatted. We force the last block to be hatted. Notice that each connected
component of $P_3^2$ is a product of permutohedra. The hexagon, a two dimensional permutohedron, also has edges and vertices labeled by $2$ colored ordered set partitions, all of color type hatted, hatted, hatted. Each of the six edges are a product of zero dimensional permutohedra, with facets having color type bald, hatted. Lastly, the disconnected vertices are all zero dimensional permutohedra, with facets of color type hatted, bald, hatted.}
\begin{tikzpicture}

\node at (4,2)[label=above:{1}-{2}-{$\hat{3}$}]{$\bullet$};
\node at (6,2)[label=above:{2}-{1}-{$\hat{3}$}]{$\bullet$};
\node at (5,1.7){12-{$\hat{3}$}};
\node at (4,0.25)[label=above:{1}-{3}-{$\hat{2}$}]{$\bullet$};
\node at (6,0.25)[label=above:{3}-{1}-{$\hat{2}$}]{$\bullet$};
\node at (5,-0.05){13-{$\hat{2}$}};
\node at (4,-1.5)[label=above:{3}-{2}-{$\hat{1}$}]{$\bullet$};
\node at (6,-1.5)[label=above:{2}-{3}-{$\hat{1}$}]{$\bullet$};
\node at (5,-1.8){23-{$\hat{1}$}};
\node at (8,2)[label=above:{1}-{$\hat{2}$}-{$\hat{3}$}]{$\bullet$};
\node at (10,2)[label=above:{1}-{$\hat{3}$}-{$\hat{2}$}]{$\bullet$};
\node at (9,1.7){1-{$\hat{2}\hat{3}$}};
\node at (8,0.25)[label=above:{2}-{$\hat{1}$}-{$\hat{3}$}]{$\bullet$};
\node at (10,0.25)[label=above:{2}-{$\hat{3}$}-{$\hat{1}$}]{$\bullet$};
\node at (9,-0.05){2-{$\hat{1}\hat{3}$}};
\node at (8,-1.5)[label=above:{3}-{$\hat{1}$}-{$\hat{2}$}]{$\bullet$};
\node at (10,-1.5)[label=above:{3}-{$\hat{2}$}-{$\hat{1}$}]{$\bullet$};
\node at (9,-1.8){3-{$\hat{1}\hat{2}$}};
\node at (-1,-1)[label=above:{$\hat{1}$}-{2}-{$\hat{3}$}]{$\bullet$};
\node at (1,-1)[label=above:{$\hat{2}$}-{1}-{$\hat{3}$}]{$\bullet$};
\node at (2.5,-1)[label=above:{$\hat{3}$}-{1}-{$\hat{2}$}]{$\bullet$};
\node at (-1,-2)[label=above:{$\hat{1}$}-{3}-{$\hat{2}$}]{$\bullet$};
\node at (1,-2)[label=above:{$\hat{2}$}-{3}-{$\hat{1}$}]{$\bullet$};
\node at (2.5,-2)[label=above:{$\hat{3}$}-{2}-{$\hat{1}$}]{$\bullet$};

\draw[thick,pattern=north west lines, pattern color=gray!80](-1,1)--(0,0)--(2,0)--(3,1)--(2,2)--(0,2)--(-1,1);
\node (1) at (-1,1){$\bullet$};
\node (2) at (0,0){$\bullet$};
\node (3) at (2,0){$\bullet$};
\node (4) at (3,1){$\bullet$};
\node (5) at (2,2){$\bullet$};
\node (6) at (0,2){$\bullet$};
\draw[thick](4,2)--(6,2);
\draw[thick](4,0.25)--(6,0.25);
\draw[thick](4,-1.5)--(6,-1.5);
\draw[thick](8,2)--(10,2);
\draw[thick](8,0.25)--(10,0.25);
\draw[thick](8,-1.5)--(10,-1.5);
\node at (1,1){\LARGE$\hat{1}\hat{2}\hat{3}$};
\end{tikzpicture}
\label{blue_red}
\end{figure}
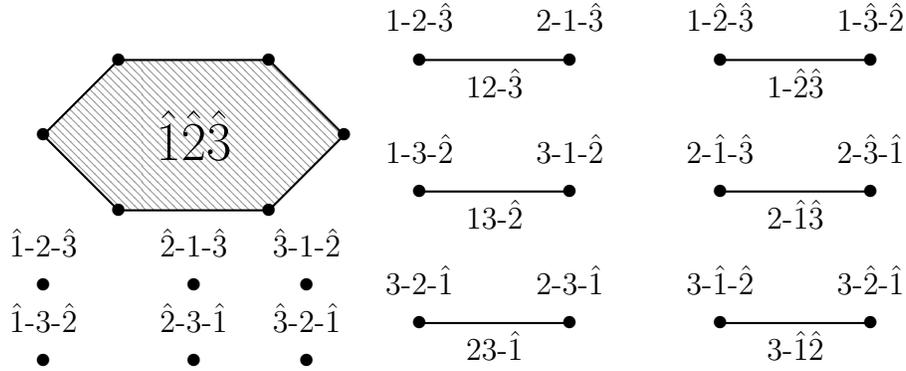

\section{$\alpha$-colored Eulerian polynomials}
\label{section_polynomial}
We continue with the narrative started in Section ~\ref{alpha_complex_section}, namely, we generalize properties of the usual permutohedron to the $\alpha$-colored permutohedron. Associated with any simplicial complex is its $h$-polynomial, the polynomial obtained by applying a certain transformation to the $f$-polynomial. It is known that the $h$-polynomial of the usual permutohedron is the Eulerian polynomial $A_n(x)$.  By extending the $h$-polynomial to polytopal complexes, we now explore the $\alpha$-colored analog for the $\alpha$-colored permutohedron, $P_n^\alpha$.

\begin{definition}
\label{h_poly}
The \emph{$h$-polynomial} of a polytopal complex $P$ of dimension $n$ is the polynomial obtained by transforming the $f$-polynomial of $P$ by $h(t) = (1-t)^{n-1}f(t/(1-t))$. 
\end{definition}
\begin{remark}
\label{remark_h_poly}
Note that an alternate definition of the $h$-polynomial of a simplicial or polytopal complex is $h(t)=f(t-1)$, for $f(t)$ the $f$-polynomial of the complex. This definition will yield the reverse of the $h$-polynomial given in Definition~\ref{h_poly}.
\end{remark}
Computing the $h$-polynomial for the $\alpha$-colored permutohedron $P_n^\alpha$ will give us a generalization of the Eulerian polynomials, which we will refer to as the $\alpha$-colored Eulerian polynomials.
\begin{definition}
\label{alpha_polynomial}
The \emph{$\alpha$-colored Eulerian polynomial}, denoted $A_n^\alpha(t)$, is defined to be the $h$-polynomial of $P_n^\alpha$.
\end{definition}

\begin{proposition}
\label{f_alpha}
The $f$-polynomial of $P_n^\alpha$ is given by the $f$ polynomial for the usual permutohedron evaluated at $\alpha\cdot t$, that is, $f(P_n^\alpha;t)=f(P_n;\alpha\cdot t)$.
\end{proposition}
\begin{proof}
Faces of dimension $n-k$ in the $\alpha$-colored permutohedron $P_n^\alpha$ are in one to one correspondence with the elements of rank $k$ in $Q_n^{\alpha}$. An $\alpha$-colored ordered set partition of rank $k$ has $n-k+1$ parts, of which $n-k$ are free to have any of $\alpha$ colors. Therefore, the number of elements of rank $k$ in $Q_n^\alpha$ equals the number of elements of rank $k$ in the usual ordered set partition lattice $Q_n$ times $\alpha^{n-k}$. Translating back to the permutohedron, this implies that the number of elements of dimension $n-k$ in $P_n^\alpha$ is $\alpha^{n-k}$ times the number of elements of dimension $n-k$ in $P_n$, and the result follows.
\end{proof}
\begin{proposition}
\label{coefficient}
The $\alpha$-colored Eulerian polynomial of order $n$ is given by $h_n(t)=\sum_{i=0}^{n-1}\gamma_i t^i$ for
$$\gamma_i=\sum_{j=n-i}^n S(n,n-j)(n-j)!\alpha^{n-j-1}\binom{j}{n-i}(-1)^{j-n+i}.$$
\end{proposition}
\begin{proof}
We have that $f(P_n^\alpha)=\sum_{i=1}^n S(n,n-i)(n-i)!\alpha^{n-i-1}x^i$, since the face lattice of $P_n^\alpha$ is the dual of $Q_n^\alpha$. Now compute $h(P_n^{\alpha})$ by evaluating the latter polynomial at $x=t-1$ and reversing the coefficients to obtain the desired result (see Remark~\ref{remark_h_poly}).
\end{proof}
Notice that when $\alpha=1$, the expression above for $\gamma_i$ reduces to the usual Eulerian number $A(n,i)$ using ~\cite[Theorem 1.18]{Bona}.

Table~\ref{table_polynomials} lists the two colored Eulerian polynomials $A_n^2(t)$ for $n$ from $2$ to $6$.
\begin{table}
\caption{\label{table_polynomials}The two colored Eulerian polynomials, $A_n^2(t)$.}
\begin{center}
 \begin{tabular}{c | c } 
$n$ & $A_n^2(t)$\\
\hline
2 & $3t+1$\\
\hline
3 & $13t^2+10t+1$\\
\hline
4 & $75t^3+91t^2+25t+1$\\
\hline
5& $541t^4+896t^3+426t^2+56t+1$\\
\hline
6& $4683t^5+9829t^4+6734t^3+1674t^2+119t+1$
\end{tabular}
\end{center}
\end{table}
\begin{proposition}
\label{closed_form_h}
When written as a sum over permutations, the $\alpha$-colored Eulerian polynomial $A_n^\alpha(t)$ satisfies
\begin{equation}
\label{closed_form}
A_n^\alpha(t):=h(P_n^\alpha)=\sum_{\pi\in\mathfrak{S}_n}(\alpha t)^{d(\pi)}(1+(\alpha-1)t)^{n-1-d(\pi)}.
\end{equation}
\end{proposition}
\begin{proof}
By Proposition ~\ref{f_alpha}, $f(P_n^{\alpha}; t) = f(P_n; \alpha\cdot t)$. So
\begin{align*}
A_n^\alpha(t)=h(P_n^\alpha,t)&=f(P_n^\alpha,t/(1-t))\cdot(1-t)^{n-1}\\
&=f(P_n,\alpha t/(1-t))\cdot(1-t)^{n-1}.
\end{align*}
Now use $\alpha\cdot t/(1-t)=\beta/(1-\beta)$ where $\beta=\alpha\cdot t/(1+(\alpha-1)t)$. Lastly apply the $f$ to $h$ transform of Definition~\ref{h_poly} once more to obtain the desired result.
\end{proof}
By inspection, we notice that the $\alpha$-colored Eulerian polynomials appear to be unimodal. We show that, in fact, they are real rooted.
\begin{proposition}
The $\alpha$-colored Eulerian polynomial $A_n^\alpha(x)$ has all real roots.
\end{proposition}
\begin{proof}
When $\alpha=1$ the result is clear. For $\alpha>1$, 
by Proposition~\ref{closed_form_h} we have that 
\begin{equation}
\label{usual_euler}
A_n^{\alpha}(t)=(1+(\alpha-1)t)^{n-1}\cdot A_n(\alpha t/ (1+(\alpha-1)t).
\end{equation}Using the real-rootedness of the Eulerian polynomials, Equation~\eqref{usual_euler} gives us that the possible roots of $A_n^{\alpha}(t)$ are $t=-1/(\alpha-1)$ or \\$t=\beta/(\alpha-\beta(\alpha-1))$, for $\beta$ a (real) root of the usual Eulerian polynomial~$A_n(t)$.
\end{proof}
We record a property, the proof of which is immediately obtained by letting $t=1$ in Equation~\eqref{closed_form}.
\begin{proposition}
\label{sum}
The sum of the coefficients of $A_n^\alpha(t)$ is given by $\alpha^{n-1}n!.$
\end{proposition}

 Proposition~\ref{sum} indicates that the $\alpha$-colored Eulerian polynomials may be defined combinatorially, in particular, as ``descents" over $\alpha$-colored permutations with a fixed last color. Section~\ref{section_colored_permutations} explores this avenue, and in particular, develops the correct notion of descent.
\section{combinatorial description for $A_n^{\alpha}$}
\label{section_colored_permutations}
In section~\ref{section_polynomial} we developed the $\alpha$-colored Eulerian polynomials as the $h$-polynomials of the $\alpha$-colored permutohedron. As the usual Eulerian polynomials are defined in terms of descents in the symmetric group, we now develop the analogous combinatorial description for the $\alpha$-colored Eulerian polynomials.

\begin{definition}[$\alpha$-colored permutation]
\label{colored_permutation}
\rm{An $\alpha$-colored permutation is a list ~$w_1^{c_1}w_2^{c_2}\dots w_n^{c_n}$ with $w_1w_2\dots w_n\in\mathfrak{S}_n$, and the tuple $(c_1,\dots,c_n)$ a list of colors where each $c_i$ is one of 
$\alpha$ colors. We denote the collection of $\alpha$-colored permutations as $\mathfrak{S}_n^\alpha$. Furthermore,
we denote the collection of $\alpha$-colored permutations with a fixed last color $\beta$ as $\mathfrak{S}_{n,\beta}^\alpha$.
}
\end{definition}
We now define descents for $\alpha$-colored permutations.
\begin{definition}
\label{def_alpha_descent}
Let $\tau$ be an $\alpha$-colored permutation. The \emph{descents} of $\tau$ are indices where
$w$ has a descent or where colors change. That is, $\Des(\tau)$ is the subset of $[n-1]$ given by 
$$\Des{(\tau)} = \{ i : w_i > w_{i+1} \,{\rm or}\,  c_i \neq c_{i+1} \}.$$ 
\end{definition}
\begin{example}
\rm{
Let $\alpha=2$ and $n=3$. Let $\tau=2^11^23^1$. Then we have the permutation $w=213$ and color vector $c=(1,2,1)$.
The descent set of $\tau$ is given by $\Des(\tau)=\{1,2\}$. We have a descent in position one since we have a descent in the permutation $w$ in position one and we have a change in color from color $1$ to color $2$. We have a descent in the second position because in moving from position two to position three we change from color $2$ to color $1$. Note that we do not double count descents, namely, even though in position $1$ we have a descent for $w$ and a color change for $c$, the index $1$ is only included once in $\Des(\tau)$.
}
\end{example}
\begin{theorem}
\label{permutation_interpretation}
The $\alpha$-colored Eulerian polynomial $A_n^{\alpha}$ can be defined over the descents of $\alpha$-colored permutations with a fixed last color $\beta$. In particular,
$$A_n^{\alpha}(x)=\sum_{\tau\in\mathfrak{S}_{n,\beta}^\alpha}x^{d(\tau)}$$
\end{theorem}
\begin{proof}
Proposition~\ref{closed_form} and the binomial theorem show the coefficient of $t^m$ in $A_n^{\alpha}(t)$ to be 
\begin{equation}
\label{coeff_t}
\sum_{\substack{\pi\in\mathfrak{S}_n\\ 0\leq d(\pi)\leq m}} \alpha^{d(\pi)}(\alpha-1)^{m-d(\pi)}\binom{n-1-d(\pi)}{m-d(\pi)}.
\end{equation}
We now show that equation~\eqref{coeff_t} also counts the number of $\alpha$-colored permutations with $m$ descents with a fixed last color $\beta$. 

An $\alpha$-colored permutation $\tau=w_1^{c_1}w_2^{c_2}\dots w_n^{c_n}$ has $m$ descents precisely when the permutation $w_1w_2\dots w_n\in\mathfrak{S}_n$ satisfies $0\leq d(w)\leq m$ and the color vector $c$ changes $m-d(w)$ times for indices not in the descent set of $w$. 

Fix a permutation $w\in\mathfrak{S}_n$ with $0\leq d(w)\leq m$. We now enumerate all $\tau\in\mathfrak{S}_{n,\beta}^\alpha$ so that $\tau=w_1^{c_1}w_2^{c_2}\dots w_n^{c_n}$ has $m$ descents.

Of the $n-1-d(w)$ positions where $w$ does not have a descent, choose $m-d(w)$ positions where our color vector $c$ will change, of which there are $\binom{n-1-d(w)}{m-d(w)}$ ways to do so. This will determine $m-d(w)+1$ runs in $c$ of the same color. The rightmost color run of $c$ will be the color of the last block, since this is a fixed color. Moving left from the last run, we have $(\alpha-1)$ color choices for each run, as each run must be a different color than the run that succeeds it to guarantee a descent is introduced. Therefore, we have $(\alpha-1)^{m-d(w)}$ choices for our color runs. Furthermore, since descents in $\mathfrak{S}_{n,\beta}^\alpha$ are not double counted, we may also change the color $c_i$ for each index $i$ where the permutation $w$ has a descent, giving us $\alpha^{d(w)}$ choices. Finally, summing over $w\in\mathfrak{S}_n$ with $0\leq d(w)\leq m$ gives the desired result.
\end{proof}
\begin{proposition}
\label{recurrence}
Let $A^\alpha(n,k)$ be the coefficient of $x^k$ in $A_n^{\alpha}(x)$, which we call the $\alpha$-colored Eulerian numbers. A recurrence relation for the $\alpha$-colored Eulerian numbers is given by:
\begin{align*}
&A^\alpha(n,k)=\\
&(k+1)A^\alpha(n-1,k)+\\
&[\alpha+n-k-1+(k-1)(\alpha-1)+(\alpha-1)(\alpha-1)!]A^\alpha(n-1,k-1)+\\
&(\alpha-1)(n-k)A^\alpha(n-1,k-2).
\end{align*}
\end{proposition}
Using Theorem~\ref{permutation_interpretation}, we now describe how to prove the recursion. We create a colored permutation in $\mathfrak{S}_{n,\beta}^\alpha$ with $k$ descents by inserting an $n$ into the appropriate permutations in $\mathfrak{S}_{n-1,\beta}^\alpha$. Care is needed is developing the recursion, as the $n$ we insert can have one of $\alpha$ colors. Additionally, the insertion of $n$ can produce two descents, hence we have a three term recursion. 

Notice that $\alpha=1$ recovers the recursion for the usual Eulerian numbers, see Theorem $1.7$ of \cite{Bona}.
\section{Acknowledgments}
The author thanks his advisor, Dr.\ Richard Ehrenborg, for assisting with this exploration. The author also thanks Dr.\ Kyle Petersen of Depaul University for his great text on Eulerian numbers~\cite{Petersen}, as well as for invaluable help via email, specifically with the combinatorial description of $A_n^{\alpha}(x)$ given in Theorem~\ref{permutation_interpretation}.
\bibliographystyle{plain}
\bibliography{alpha_colored}

\begin{thebibliography}{1}

\bibitem{Bona}
Mikl{\'o}s B{\'o}na.
\newblock {\em Combinatorics of permutations}.
\newblock Discrete Mathematics and its Applications (Boca Raton). CRC Press,
  Boca Raton, FL, second edition, 2012.
\newblock With a foreword by Richard Stanley.

\bibitem{Motzkin}
Theodore~S. Motzkin.
\newblock Sorting numbers for cylinders and other classification numbers.
\newblock In {\em Combinatorics ({P}roc. {S}ympos. {P}ure {M}ath., {V}ol.
  {XIX}, {U}niv. {C}alifornia, {L}os {A}ngeles, {C}alif., 1968)}, pages
  167--176. Amer. Math. Soc., Providence, R.I., 1971.

\bibitem{Petersen}
T.~Kyle Petersen.
\newblock {\em Eulerian numbers}.
\newblock Birkh\"auser Advanced Texts: Basler Lehrb\"ucher. [Birkh\"auser
  Advanced Texts: Basel Textbooks]. Birkh\"auser/Springer, New York, 2015.
\newblock With a foreword by Richard Stanley.

\bibitem{EC1}
Richard~P. Stanley.
\newblock {\em Enumerative combinatorics. {V}olume 1}, volume~49 of {\em
  Cambridge Studies in Advanced Mathematics}.
\newblock Cambridge University Press, Cambridge, second edition, 2012.

\end{thebibliography}

\end{document}